\documentclass[12pt]{amsart}
\usepackage{hwstyle}

\newcommand{\len}{\ell}

\renewcommand{\G}{\mathbf{G}}

\newcommand{\I}{\mathrm{Id}}

\author{ Delaram Kahrobaei}
\address{Delaram Kahrobaei, University of York, Department of Computer Science, United Kingdom, The City University of New York, CUNY Graduate Center, PhD Program in Computer Science, New York, USA, New York University, Computer Science and Engineering Department.}
\email{delaram.kahrobaei@york.ac.uk, dk2572@nyu.edu, dkahrobaei@gc.cuny.edu}

\author{Keivan Mallahi-Karai}
\address{Jacobs University, Department of mathematics and logistic, 28759, Bremen, Germany.}
\email{k.mallahikarai@jacobs-university.de}
\title{Some applications of arithmetic groups in cryptography}
 
\begin{document}
\maketitle

\begin{abstract}
In this paper we will offer a new symmetric-key cryptographic scheme which is based on the existence of exponentially distorted subgroups in arithmetic groups. Aside from this, we will also provide new examples of distorted subgroups in $\SL_n(\ZZ[x])$ which can be utilized for the same purpose. 

\end{abstract}
%\tableofcontents

\section{Introduction}
%Most constructions of public-key crypto-systems typically hing upon one-way functions, that is, functions that are easy to compute, while finding the inverse images are prohibitively hard. 

Since the pioneering work of Anshel-Anshel-Goldfeld, group theory has proved to be a rich source of platforms for cryptographic primitives. While the initial constructions were based on abelian groups (e.g. additive and multiplicative group of finite fields), more recently various non-abelian groups have also emerged as effective tools for such constructions.   

Various concrete groups have been proposed as platforms for various encryption methods. For instance,  braid groups constructed by Emil Artin  \cite{Artin}
are examples of such groups based on which Anshel, Anshel and Goldfeld \cite{AAG}, and, independently, Ko, Lee, Cheon, Han, Kang and Park \cite{KLCHKP} constructed two cryptosystems. Note that in these constructions the difficulty of decryption relies on difficulty of solving a specific problem in the undellying group. For instance, in the case of braid groups, it is the conjugacy search problem that will make the platform useful for cryptography. Linear groups were proposed by Baumslag-Fine-Xu \cite{BFX}, this proposal is also proposes a symmetric-key encryption using matrices.

In this paper we focus on symmetric-key cryptographic scheme in section \ref{crypto2}, which use the idea of decoy. This idea has been discussed in series of series of paper by Cavallo, Di Crecsenzo, Kahrobaei, Khodjaeva, Shpilrain \cite{CDDS, DKKS1} on secure delegation of computation where decoys are used in a different context. In our context, the idea appears to be new and, moreover, yields some interesting questions in group theory.

In this paper, we will turn to class of groups, namely, arithmetic groups, that have not yet attracted as much attention in group-based cryptography. Roughly speaking, to construct an arithmetic group, one starts with a linear algebraic group defined over the field of rational numbers, and considers the elements of the group with integer entries. Recall that a linear algebraic group $\G$ is a subgroup of the group of $n \times n$ invertible matrices that is defined by the zero set 
of a finite number of polynomials. When these polynomials have rational coefficients, we say that $\G$ is defined over the field of rational numbers. 

Let us explain this by an example: consider the special linear group $\SL_n$ defined by the condition $\det(X)=1$. Then the associated arithmetic group is the group
$\SL_n(\ZZ)$ consisting of $n \times n$ matrices with integer entries. Arithmetic groups have been the subject of many studies. For a basic introduction to arithmetic groups, we refer the reader to \cite{Borel2}. The aspect of arithmetic groups that will be useful for us is the existence of U-elements in them. Roughly speaking, an element $g$ of a finitely generated group $G$ is called a U-element if for all $n \ge 1$ the element $g^n$ can be expressed as a word of length $ C \log( |n|+1)$ in terms of generators of the group. Arithmetic groups tend to have an abundance of unipotent elements. In fact, it's a celebrated theorem of Lubotzky, Mozes, Raghunathan \cite{LMR} that under some (necessary) conditions on the arithmetic group $G$, U-elements in $G$ are precisely the virtually unipotent elements. 
In this direction, we will prove the following theorem, which, among other things, will show provide a large number of ``arithmetic-like'' groups also have a wealth of U-elements, and hence can be used for our cryptographic platform. More precisely, we will show

\begin{theorem}\label{zxb} Let  $n \ge 3$, and $m \ge 1$. An element of the  group $\SL_n(\ZZ[x_1, \dots, x_m])$ is a U-element iff it is virtually unipotent. 
\end{theorem}

For the definition of U-elements, its history and relevant results, as well as the proof of Theorem \ref{zxb}, we will refer the reader to section 
\ref{universal}.

Our goal in this paper is to show how one can exploit properties of these groups to construct new cryptosystems. 
We propose a new platform for the private key encryption schemes developed by Chatterji, Kahrobaei, Lu. They use the fact that there are exponentially distorted subgroups in certain groups that the geodesic length problem is in polynomial time. In particular they propose certain classes of hyperbolic group. In this paper we propose that using distortion of the unipotent subgroups in lattices in higher rank Lie groups similar platforms can be constructed.

\section{The Cryptosystems Using Subgroup Distortion}\label{crypto}
The following cryptosystems have been proposed in \cite{CKL} who suggested using Gromov hyperbolic groups. Here we put forward new classes of discrete group for this purpose.
\subsection{The protocol I: basic idea} Assume that Alice and Bob would like to communicate over an insecure channel. Let $G=\langle g_1, \cdots ,g_l |R \rangle$ be a public group such that solving the geodesic length problem in $G$ is possible in polynomial time. Assume that $G$ has a large number of distorted subgroups. Let $H= \langle t_1,\cdots, t_s \rangle \subset\langle g_1, \cdots ,g_l  \rangle =G$ be a secret subgroup of $G$, that is distorted and shared only between Alice and Bob. 
\begin{enumerate}
	\item Suppose that the secret message is a integer $n \ge 1$. Alice picks $h\in H$ with $\ell_H(h)=n$, expresses $h$ in terms of generators of $G$ with $\ell_G(h)=m \ll n$ and sends $h$ to Bob.
	\item Bob then converts $h$ back  in terms of generators of $H$ and computes $\ell_H(g)=m$ in polynomial time according to assumption to recover $n$.
\end{enumerate}
\subsubsection{Security} Although $H$ is not known to anyone except to Alice and Bob and $h$ being sent with length $m\ll n$ gives infinitely many possible guesses for the eavesdropper Eve, the security of the scheme is weak since Eve will eventually intercept sufficiently many elements of $H$ to generate $H$ (one can think of the group ${\mathbb Z}$ of integers, it is enough to intercept two relatively prime integers to generate the whole group). 

\subsection{The protocol I: secure version}\label{crypto2}
In order to make it difficult for Eve to identify the subgroup $H$, suppose that Alice also sends along with $h$ occasionally elements of $G$ that do {\it not} belong to $H$. To determine how Bob can tell which elements belong to $H$ to retrieve the correct message, we will consider below the subgroup membership problem and the random number generator.
\subsubsection{Subgroup membership problem}
Suppose we have a group in which the subgroup membership problem can be solved efficiently. Then we will send some random words and the receiver first checks whether each word belongs to $H$ and then computes its length.\\

Protocol: Let $G = \langle g_1,\cdots,g_l \vert R_G\rangle$ be a group that is known to the public such that solving the geodesic length problem is possible in polynomial time and that $G$ has an abundance of distorted subgroups. Let $H= \langle h_1,\cdots,h_s \rangle $ be a shared secret subgroup of $G$ that is exponentially distorted. Assume that the subgroup membership problem in $G$ efficiently solvable.
\begin{enumerate}
	\item Alice picks $h\in H$ with $\ell_H(h)=n$, expresses $h=g_1 \cdots g_{m}$ in terms of generators of $G$ with $\ell_G(h)=m \ll n$. She randomly generates $a_0,\dots,a_m \in G \setminus H$  and sends these words to Bob.
	\item Since Bob knows the generating set for $H$, he finds $h \in H$ (since he could check the subgroup membership problem efficiently) he only uses $h\in H$, and expressed it in terms of generators of $H$ and computes $\ell_H(g)=n$ in polynomial time using the assumption for $G$ to recover $n$.
\end{enumerate}

\section{U-elements in lattices in higher rank Lie groups}
In this section, we will review some of the basic properties of those matrix groups that were used in Section \ref{crypto} for constructing the cryptosystems. As indicated there, the construction is based on existence of pairs $(G,H)$, consisting of a finitely generated group $G$, and a
subgroup $H$ of $G$ with the following properties:

\begin{enumerate}
\item There exists a constant $C_1>0$ such that for every $h \in H$, one can compute in a polynomial time a path from the identity to $h$ whose length is bounded from above by $C_1 \ell_G(h)$.
\item The membership problem in $H$ can be solved in polynomial time. 
\item $H$ is exponentially distorted in $G$, that is, there exists a constant $C_2>0$ such that for all $h \in H$ we have 
\[\ell_G( h) \le C_2 \log (1+ \ell_H(h) ). \]
\item $H$ has many conjugates in $G$. More precisely, the index of the normalizer $N_G(H)$ in $G$ is infinite. 
\end{enumerate}

Before we proceed, let us review some definitions. Let $G$ be a finitely generated group
with a fixed generating set $S$. For every $g \in G$, denote by $\len_{G,S}(g)$, or simply $\len_G(g)$ when $S$ is implicitly fixed, the least integer $k \ge 0$ such that $g$ can be expressed as a product $g= g_1^{e_1} \dots g_k^{e_k}$, where
$g_i \in S$ and $ e_i = \pm 1$ for all $ 1 \le i \le k$. The number $\len(g)$ is called the length of element $g$ with respect to the generating set $S$. One can easily show that if $S_1$ and $S_2$ are two finite generating sets for $G$,  then $ \ell_{G,S_1}(g) \le C \ell_{G,S_2}(g)$ for some constant $C>0$ and all $g \in G$. This simple fact shows that the choice of $S$ is immaterial for our discussion. 

We will now discuss some examples of groups that will appear in our construction. 
Let us start with the most prominent example, namely, the
special linear group $\SL_n(\ZZ)$, which consists of all matrices of determinant $1$ with entries in $\ZZ$.
A subgroup $H$ of $\SL_n(\ZZ)$ is called unipotent if it consists of unipotent matrices, that if for every $u \in H$ all the 
eigenvalues of $u$ are equal to $1$.

Let $H_n$ denote the subgroup of $\SL_n(\ZZ)$ consisting of all 
upper-triangular matrices with $1$ on the diagonal with all the off-diagonal entries which are not on the first row are equal to zero.
A typical element of $H_n$ has the following form:
\[ h= \begin{pmatrix}
1  & 0  & 0 &\dots & h_{1n} \\
 0 & 1  &  0 & \dots & h_{2n}\\
\vdots  & \vdots & \vdots & \vdots &   \vdots \\
0 & 0 & 0 & 1  & h_{n-1,n} \\
0 & 0 & 0 & 0 &1 
\end{pmatrix} \]

Note that $H_n$ is an abelian algebraic subgroup of $\SL_n(\ZZ)$ defined by linear equations
\[ x_{ ii}-1=0, \quad 1 \le i \le n,  \qquad x_{ij}=0, \quad  2 \le i \neq j \le n. \]

A simple corollary of this fact is that that the membership problem for an element $h \in \SL_n(\ZZ)$ in $H_n$ can be solved in polynomial time in size of entries of $h$. Let $g \in \SL_n(\ZZ)$ be an arbitrary element, and consider the conjugate subgroup 
$g^{-1}H_n g$. Note that that $g^{-1}H_ng$ is also an algebraic group, and its defining equations can be obtained by transforming the set of equations given above for $H$ using conjugation with $g$. In particular, the membership problem for all conjugates of $H_n$ can also be solved in polynomial time. 

%In fact, it will be more convenient to work with another subgroup of $G$. This subgroup, which is itself as subgroup
%$H_n$ and will be denoted by $A_n$, consists of those matrices in $H_n$ such that, aside from diagonal entries,  has only non-zero entries in their last column, i.e. a typical element of $A_n$ has the form:
%
%\[ a= \begin{pmatrix}
%1  & 0  & 0 &\dots & h_{1n} \\
% 0 & 1  &  0 & \dots & h_{2n} \\
%\vdots  & \vdots & \vdots & \vdots &   \vdots \\
%0 & 0 & 0 & 1  & h_{n-1,n} \\
%0 & 0 & 0 & 0 &1 
%\end{pmatrix} \]

\subsection{Distortion}
Suppose $H$ is a finitely generated subgroup of a finitely generated group $G$. Let $S_H$ and $S_G$ denote two fixed generating sets for $H$ and $G$ respectively. Recall that $H$ 
is exponentially distorted in $G$ if there exists a constant $C>0$ such that for each element $h \in H$, we have 
\[ \len_G(h) \le C \log(1+ \len_H(h)). \]
In other words, the geodesic joining the identity element to $k$ in $G$ is shorter by an exponentially large factor than the one in $H$ itself. An element $u \in G$ is called a U-element if the cyclic subgroup $H= \langle u \rangle$ is infinite and exponentially distorted in $G$. In other words, $g$ is a U-element if 
\[ \len_G(g^n) \le C \log (1+ |n|) \]
holds for all $n \ge 1$, and some $C>0$. It is easy to see that the definition of $U$-elements does not depend on the choice of the generating set $S$. Let us recall some examples of exponential distortion.

\begin{example}
Let $G$ denote the Baumslag-Solitar group $B(1,2)$, generated by elements $t$ and $a$ subject to the single relation $t^{-1}at=a^2$. Denote by $S$ the generating set consisting of $a^{\pm 1}$ and $t^{ \pm 1}$, and by $K$ the infinite cyclic group generated by $a$. It is easy to see that $\ell_K(a^n)$ grows linearly with $n$. On the other hand, since $t^{-k}a t^k= a^{2^k}$, we have 
$\ell_G(a^{2^k}) \le 2k+1$. By expressing an arbitrary integer $n$ in base $2$, one can easily show that the bound $ \ell_G(a^n) = O( \log n)$ holds for all $n \ge 2$.  

This group can indeed be realised as a linear group. To see this, set 
\[ t = \begin{pmatrix}
2  & 0    \\
 0 & 1 \\
\end{pmatrix}, \qquad a= \begin{pmatrix}
1  & 1    \\
 0 & 1 \\
\end{pmatrix} \]
It is easy to see that the pair $(t,a)$ satisfy the relation $ tat^{-1}=a^2$, and hence one can write $a^n$ as a word of length $O( \log n)$ in terms of $a$ and $t$. 
\end{example}

A celebrated theorem of Lubotzky, Mozes, and Raghunathan \cite{LMR} classifies U-elements for the class of $S$-arithmetic groups in higher rank. This class includes groups such as $\SL_n(\ZZ)$ for $n \ge 3$. As we will be mostly interested in applications, we will only mention a special case of their result. For convenience of the reader we will also provide a proof of a special case which also goes back to \cite{LMR}.

\begin{theorem}[Lubotzky, Mozes, Raghunathan]
For $n \ge 3$, an element $g \in \SL_n(\ZZ)$ is a U-element if and only if $g$ is virtually unipotent, that is, there exists $k \ge 1$ such that $g^k$ is unipotent. 
\end{theorem}

For a matrix $A \in \SL_{n-1}(\ZZ)$, and a vector $v \in \ZZ^{n-1}$, write
\[ M(A,v)=  \begin{pmatrix}
A  & v    \\
 0 & 1 \\
\end{pmatrix}  \]

A simple computation shows that 
\begin{equation}\label{conjugate}
M(A,0) M(\I,v) M(A,0)^{-1}= M(\I,Av). 
\end{equation}
where throughout the the article $\I$ denotes the identity matrix of the appropriate size. 

\begin{proposition}\label{}
For every vector $v \in \ZZ^{n-1}$, the length of  the element $M(I,v)$ in $\SL_n(\ZZ)$ is bounded by $C \log (1+ \| v \|)$
for some constant $C$, which only depends on $n$. Moreover, there exists a constant $C_1$ such that a path of length 
$C \log (1+ \| v \|)$ from the identity element to $g$ can be constructed in polynomial time in $ \| v \|$.

\end{proposition}

\begin{proof} We will give a proof of the bound $ \ell_S( M(I, v) ) \le C \log (1+ \| v \|)$. It is a direct corollary of the proof that a path connecting $e$ to $  M(I, v)$ can also be found in polynomial time in $ \| v \|$. It is easy to see that the statement will follow the special case of $n=3$. So, let us assume that $n=3$ and consider
a matrix $M(I, v)$ for some $ v\in \ZZ^2$. Let $e_1, e_2$ denote the standard basis for $\ZZ^2$, and consider the matrix 
\[ A= \begin{pmatrix}
1  & 1    \\
 1 & 0 \\
\end{pmatrix}. \]
One can check that  $A$ has two real eigenvalues $ \lambda_1>1$ and $-1< \lambda_2<0$ given by
\[ \lambda_1= \alpha, \qquad \lambda_2= - \alpha^{-1}, \]
where $ \alpha= \frac{1+ \sqrt{5}}{2}$ and $ \alpha^{-1}=  \frac{ \sqrt{5}-1}{2}.$ Denote the corresponding eigenvectors
by $v_1$ and $v_2$. These are given by
\[ v_1= \begin{pmatrix}  \alpha    \\
 1 \\
\end{pmatrix} , \qquad v_2= \begin{pmatrix} - \alpha^{-1}    \\
 1 \\
\end{pmatrix}. \]

Note that $v_1$ and $v_2$ are linearly independent, and for each integer 
$k \ge 1$ we have 
\[ A^k v_1 = \lambda_1^k v_1, \qquad A^{-k} v_2= \lambda_2^{-k} v_2. \]
It will be desirable to work with vectors with integral coordinates. Hence, we set 
\[ w_1=  \frac{1}{ \sqrt{5}} ( \alpha^{-1}v_1+ \alpha v_2)= \begin{pmatrix} 0   \\
 1 \\
\end{pmatrix} , \qquad w_2= \frac{1}{ \sqrt{5}} (v_1-v_2)= \begin{pmatrix} 1    \\
 0 \\
\end{pmatrix}. \]
Since $ | \lambda_1|>1$ and $| \lambda_2|<1$, we can easily see that for large values of $ k$ we have
\[ \| A^k w_1 \| =  \| \lambda_1^k v_1 + \lambda_2^k v_2 \| \approx  \lambda_1^k \|  w_1 \|. \]
Similarly, 
\[ \| A^{-k} w_2 \| =  \| \lambda_1^{-k} v_1 + \lambda_2^{-k} v_2 \| \approx  \lambda_2^{-k} \| w_2 \|. \]

These, in particular, show that  the vectors $w_1$ and $w_2$ expand exponentially under, respectively, positive and negative powers of $A$. From here one can readily see that every vector in $\ZZ^2$ is at a bounded distance from the set 
\[ \left\{ \sum_{k=-m}^{-1}  \beta_k A^{k} w_2+  \sum_{k=0}^{m}  \beta_k A^k w_1: \beta_i \in \{ 0,1,2 \}, m \ge 1 \right\}. \]
and that $m$ the value of $m$ can be chosen to be of the order $C \log (1+ \| v\|)$ for some constant $C$. The claim will now follow immediately from \eqref{conjugate}. 

\end{proof}

%
%\begin{lemma}\label{}
%For every vector $ u \in \ZZ^n$, there exist a vector $u_0$ with $\| u_0 \| \le 2$ and non-negative integers $m_1< m_2< \dots, m_r$ and $n_1< n_2< \dots< n_s$ such that 
%\[ u = u_0+  \sum_{i=1}^{r} A^{ m_i} w_1+  \sum_{j=1}^{s} A^{-n_i} w_2. \]
%\end{lemma}
%
%\begin{proof} Let us define the Fibonacci sequence by $F_{-1}=0, F_{0}=1$, and $F_{n}= F_{n-1}+ F_{n-2}$ for
%$n \ge 1$. It is not difficult to see that 
%\[ A^k= \begin{pmatrix}
%F_k  & F_{k-1}    \\
% F_{k-1} & F_{k-2} \\
%\end{pmatrix}, \qquad \qquad A^{-k}= (-1)^k \begin{pmatrix}
%F_{k-2}  & -F_{k-1}    \\
% - F_{k-1} & F_{k} \\
%\end{pmatrix}  \] 
%Consider the set of vectors of the form 
%\[ \sum_{i=1}^{r} A^{ m_i} w_1+  \sum_{j=1}^{s} A^{-n_i} w_2, \]
%where $ 0 \le m_1< m_2< \dots< m_r$ and $ 0 \le n_1< n_2< \dots < n_s$. Thus the problem of expressing $u$ in the desired form is equivalent to 
%\[ u=  \sum_{i=1}^{r}  \begin{pmatrix} F_{m_i}    \\
% F_{m_{i}-1} \\
%\end{pmatrix} +   \sum_{j=1}^{s} \begin{pmatrix} (-1)^{j+1} F_{n_{j-1}}    \\
% F_{n_j} \\
%\end{pmatrix}. \]

Let $H$ denote the abelian subgroup of $\SL_n(\ZZ)$ consisting of all matrices of the form $M(\I, v)$ with $ v \in \ZZ^{n-1}$. In the next proposition we will show that $N_G(H)$ has infinite index in $G$

\begin{proposition}\label{}
The normalizer of $H$ consists of matrices of the form 
\[ \begin{pmatrix}
B_{n-1 \times n-1}  & x_{n-1 \times 1}    \\
 0_{1 \times n-1} & \epsilon \\
\end{pmatrix} \]
where $B \in \GL_n(\ZZ)$,  $ \epsilon = \pm 1$ and $\det B= \epsilon$, and has infinite index in $G$. 
\end{proposition}

\begin{proof}
Consider a matrix $C$ in the normalizer of $H$ in $\SL_n(\ZZ)$, and assume that it is partitioned into blocks as
\[ A= \begin{pmatrix}
B  & x    \\
 y^t & \epsilon \\
\end{pmatrix} \]
where $B$ is $n -1$ by $n-1$, $x, y \in \ZZ^n$, $ \epsilon \in \ZZ$ and $t$ denotes the transpose.  This implies that for every vector $ a \in \ZZ^n$, there exists a vector $ a' \in \ZZ^n$ such that 
\[ \begin{pmatrix}
B  & x  \\
 y^t & \epsilon \\
\end{pmatrix} \    \begin{pmatrix}
\I  & a    \\
 0 & 1 \\
\end{pmatrix} = \begin{pmatrix}
\I  & a'    \\
 0 & 1 \\
\end{pmatrix} \begin{pmatrix}
B  & x    \\
 y^t & \epsilon \\
\end{pmatrix}. \]
\end{proof}
Comparing the entries $(n,n)$ of the products yields for every $a \in \ZZ^n$:
\[ y^t a+ \epsilon= \epsilon \]
which implies that $ y=0$. From here it follows easily that $ \epsilon \det B =1$. Since $\det B \in \ZZ$ and  $ \epsilon \det B =1$, we must have
$\det B = \epsilon \in \{ \pm 1 \}$. This proves the claim. It is now obvious that this subgroup has infinite index in $\SL_n(\ZZ)$.

\begin{proposition}\label{}
Every conjugate of $A$ is also exponentially distorted in $\SL_n(\ZZ)$. 
\end{proposition}

\begin{proof} Let $S_A$ denote a generating set for $A$, and for $g \in G$, 
set $S_g= \{g a g^{-1} : a \in S_A \}$.
Consider  an element $x \in H^g:=g^{-1}Ag $, and write $x= g^{-1}ag$ for some $a \in H$. Since $H$ is exponentially distorted in $\SL_n(\ZZ)$, we have
\[ \ell_G(a) \le C \log ( 1+ \ell_H(a)). \]
This implies that 
\[ \ell_G(x) \le C \log ( 1+ \ell_{H^g}(x ) ).  \]
where $\ell_{H^g}$ is the word metric on $H^g$ defined with respect to the generating set $S_{g}$. 
This proves the claim.
\end{proof}

From here we can see that any one of the conjugates of $H$ can be used instead of $H$. The versatility gained in this way allows us
to replace $H$ by its conjugates, which are all isomorphic to $H$, but are harder to detect.

\section{Non-arithmetic linear groups with U-elements}\label{universal}
In this section we will construct a new class of linear groups with a wealth of U-elements that have the desired properties for the encryption scheme we have before. Let $R_m:= \ZZ[x_1, \dots, x_m]$ denote the ring of polynomials with integer coefficients in variables $x_1, \dots, x_m$.  Denote by $\SL_n(\ZZ[x_1, \dots, x_m])$ the group consisting of $n \times n$ matrices of determinant $1$ with entries from $R_m$. Note that if $R$ is any $\ZZ$-ring generated by $m$ elements, then $R$ can be viewed as a quotient of $R_m$. Note that any ring homomorphism $ \phi: R_m \to R$ induces a group homomorphism from $\SL_n(R_m)$ into $\SL_n(R)$. It is important to remark that this homomorphism need not be surjective. For simplicity of notation, we will prove Theorem \ref{zxb} in the special case of $m=1$. It can be easily seen that the same proof works for general $m \ge 1$. More precisely, we will show
 
\begin{theorem}\label{zx}
For $n \ge 3$, an element $u$ of the group $G=\SL_n(\ZZ[x])$ is a U-element iff $u$ is virtually unipotent.
\end{theorem}

Before starting the proof we will need a lemma about virtually unipotent matrices in $\SL_n(\ZZ)$:

\begin{lemma}\label{virtual}
Suppose $n \ge 2$, and $g \in \SL_n(\ZZ)$ is a virtually unipotent matrix. Then there exists $r=r(n)$ only depending on $n$ such that 
$(g^r-\I)^n=0$.
\end{lemma}

\begin{proof}
Suppose $g \in \SL_n(\ZZ)$ 
is a virtually unipotent matrix. Since $g^m$ is unipotent for some $m \ge 1$, we know that every eigenvalue of $g$ is a root of unity.
On the other hand, the characteristic polynomial of $g$ has degree at most $n$, implying that any eigenvalue $ \lambda$ of $g$ is a root 
of polynomial equation of degree at most $n$ with integer coefficient. In view of the fact that that if degree of a $k$-th root of unity is $\phi(k)$, 
and the simple fact that $\phi(k) \to \infty$ as $k \to \infty$, it follows that only finitely many roots of unity can be potentially eigenvalues of $g$. In particular, there exists an integer $r \ge 1$ depending only on $n$ and not on $g$ such that $ (g^r - \I)^n=0$. 
\end{proof}

\begin{proof}[Proof of Theorem \ref{zx}] Let us first show that if $u$ is a U-element then $u$ is virtually unipotent.  
For each integer $k \in \ZZ$, let ${\mathrm{ev}}_k: \ZZ[x] \to \ZZ$ denote the evaluation homomorphism at point $k$, mapping $f(z) \in \ZZ[x]$ to
$f(k)$. Denote by $\pi_k: \SL_n(\ZZ[x]) \to \SL_n(\ZZ)$ the group homomorphism which is induced by ${\mathrm{ev}}_k$.  Since $u$ is a U-element, it follows from Lemma \ref{quotient} that $\pi_k(u)$ is also a U-element, and hence virtually unipotent. From the above remark it follows that $ ( \phi_k(u)^r- \I)^n=0$. This means that for each entry of 
the matrix $(u^r- \I)^n$ (which is a polynomial in $x$) when evaluated at an arbitrary integer is zero. This implies that 
$(u^r- \I)^n=0$, which shows that $u$ is virtually unipotent.

Let us now prove the converse statement. Let us denote the elementary matrix with entry $t$ in row $i$ and column $j$ by $E_{ij}(t)$. Fix a generating set $S$ for $\SL_n(\ZZ[x])$ consisting of all elementary matrices of the form $E_{ij}(1)$ and $E_{ij}(x)$. Note that using the commutator relations all the elementary 
matrices $E_{ij}(x^r)$ for $ r \ge 2$ can also be generated.  Recall that $u$ is unipotent when $u-\I$ is nilpotent. It is well-knowns that $F$ is a field and $w \in \GL_n(F)$ is a nilpotent matrix, then $ w^n=0$. Writing $u=\I+w$, from $w^n=0$, we have for every $k \ge n$:
\[ u^q= (\I+w)^q=  \sum_{k=0}^{n-1} {q \choose k}  w^k. \]
In particular it follows that each entry of $u^q$ is a polynomial of degree at most $n-1$ in the polynomial ring $\QQ[x][q]$. Let $d$ be a common denominator of the fractions appearing the entries of $u^q$. 
At this point, we are going to use the following special case of a theorem of Suslin proven in \cite{Suslin}:

\begin{theorem}[Suslin]
For $n \ge 3$, and $m \ge 1$ the group $\SL_n(\ZZ[x_1, \dots, x_m])$ is generated by elementary matrices. 
\end{theorem}

We can now view the matrix $u^{dq}$ for $q \ge 1$ as an element of the group $\SL_n(\ZZ[q,x])$. Using Suslin's Theorem, $u^{dq}$ can be expressed as a product of elementary matrices $\I+ E_{ij}( t)$, with $t \in \ZZ[x,q]$.
Write $t= \sum_{l=0}^{d} a_l q^l$, with $a_l \in \ZZ[x]$. Choose an index $k$ different from $i,j$ and note that 
\[ E_{ij}( t)= \prod_{l=0}^{d} E_{ij}( a_l q^l)=  \prod_{l=0}^{d} [E_{ik}( a_l), E_{kj}(q^l)]. \]
This implies that 
\[ \ell_S( E_{ij}(t) ) \le 2 \sum_{l=0}^{d} \ell_S( E_{ik}(a_l)) +\ell_S(E_{kj}(q^l)) \le C \log q+ O(1). \] 
Since $u^q$ is a product of a bounded (independent of $q$) of elementary matrices of the above form, it has a length
bounded by $O( \log q)$. It follows that $u$ is a U-element. 
\end{proof}

Many other examples can be built out of this example using the following lemma:

\begin{lemma}\label{quotient}
Let $G$ be a group generated by a set $S$. Suppose $ g \in G$ is such that the cyclic subgroup generated by $g$ 
is exponentially distorted in $G$. Let $ \phi: G \to H$ be a group homomorphism. If $\phi(g)$ has infinite order, then it is exponentially distorted in $H$. 
\end{lemma}

\begin{proof}
Without loss of generality we can assume that $\phi$ is surjective. Let $S$ be a finite generating set for $G$ and 
$ \overline{S} = \phi(S)$ be the image of $G$ under $\phi$. One can readily see that $ \overline{ S}$ is a generating 
set for $H= \phi(G)$, and that for every $g \in G$, we have
\[ \ell_H( \phi(g) ) \le \ell_G( g). \]
From here it follows immediately that if $g$ is a U-element, then so is $\phi(g)$. 
\end{proof}

\begin{example} Using Theorem \ref{zx} and Lemma \ref{quotient} one can construct a large number of other matrix groups with U-elements. Let $I$ be an ideal of $\ZZ[x]$. For instance, we can choose a polynomial $f(x) \in \ZZ[x]$, and let $I$ be the ideal generated by $f(x)$. Set $R= \ZZ[x]/I$. It is easy to see from Theorem \ref{zx} and Lemma \ref{quotient} that 
the image of every unipotent matrix in $\SL_n(\ZZ[x]/I)$ is a U-element. 

Let us consider a special case of interest here. Let $I_k$ be the ideal generated by $f(x)=x^{k+1}$ where $k \ge 1$. 
Consider the group $G_{n,k}=\SL_n(\ZZ[x]/( x^{k+1}))$. One can easily check that  $G_{n,k}$ is a finitely generated group, with a generating set consisting of all elementary matrices of the form $E_{ij}(x^r)$ with $ 0 \le r \le k$. Let $\pi_k$ denote the reduction map from $R/I_k$ to $\ZZ$ defined by $f(x) \mapsto f(0)$. It is easily seen that $\pi_k$ induces a surjection from $\SL_n(\ZZ[x]/I_k)$ to $\SL_n(\ZZ)$, yielding the following exact sequence:
\[ 0 \mapsto N_{n,k} \to\SL_n(\ZZ[x]/I_k) \to \SL_n(\ZZ) \to 0 \]
where $N_{n,k}$ denotes the kernel of the reduction map. 

\begin{proposition}\label{kernel}
The kernel $N_{n,k}$ is a nilpotent group of class $k$. 
\end{proposition}

\begin{proof}
It is easy to see that $N_{n,k}$ consists of those matrices $A \in G_{n,k}$ such that every entry of  $A-\I$ is divisible by $x$. Suppose $A,B \in N_{n,k}$. Then 
$A= \I+ x A'$ and $B= \I+ x B'$, where $A',B'$ are matrices with integer coefficients. Note that 
\[ AB= \I + x(A'+ B')+ xA'B', \qquad BA= \I+ x(A'+B') + x B'A'. \]
In particular, we have $AB \equiv BA \pmod{I^2}$, which shows that $[A,B] \equiv \I \pmod{I^2}.$ By a similar argument one can show that if $A \equiv \I \pmod{I^p}$ and $B \equiv \I \pmod{I^q}$, then $[A,B] \pmod \I \pmod{I^{p+q} }$. The claim follows from here. 
\end{proof}
\end{example}

\section{KPA security} In this section we will turn to the question of security of the scheme. 
It is well known that the main security property of a  symmetric-key encryption scheme is KPA security. This poses the question of whether the adversary will be able to recover the secret generating set if she accumulates several elements of a known geodesic length with respect to an unknown generating set. In this direction we will prove the following

\begin{proposition}\label{}
Let $\Gamma$ be an infinite finitely generated subgroup of $\GL_n(\CC)$. Given $g_1, \dots, g_m \in \Gamma$, and positive integers $k_1, \dots, k_m$, suppose   that there exists a generating set $S$ for $\Gamma$ such that $ \ell_{\Gamma, S}(g_i)=k_i$ for all $1 \le i \le m$. Then there are infinitely many generating sets $S_j$, $j =1,2, \dots$ such that
\begin{enumerate}
\item $$ \ell_{\Gamma, S_j}(g_i)=k_i$$ for all $ 1 \le i \le m$, and $j \ge 1$.
\item The associated metrics $\ell_{\Gamma, S_j}$ on $\Gamma$ are pairwise distinct. 
\end{enumerate}
\end{proposition}

\begin{proof} Since $\Gamma$ is fixed, in interest of conciseness we will drop it from the subscript. 
Let $S= \{ s_1, \dots, s_r \}$. First we will construct an infinite number of generating sets $S_j$ that fulfill condition (1). 
In fact, the constructed family will be of the form $S_j= S \cup \{ t_j \}$ for some $t_k \in \Gamma$ with this property. 
Let $F(x_1, \dots, x_r, y)$ denote the group generated freely by the variables $x_1, \dots, x_r, y$. For each $1 \le  i \le m$, and each word $w \in 
F(x_1, \dots, x_r, y)$ with word length (with respect to $x_1, \dots, x_r, y$)
less than $ k_i$, consider the set 
\[ A(i,w)= \{ h \in G: w(s_1, \dots, s_r, h)= g_i \}. \]
Clearly the set 
\[ A_i= \bigcup_{ \ell(w) < n_i} A(i,w) \]
consists of those $h \in H$ with the property that the length of $g_i$ with respect to the extended generating set $S'= S \cup \{ h \}$ is less than $n_i$. 
Finally, set $A= \bigcup_{1 \le i \le m} A_i$. It suffices to show that $G-A$ is infinite. To see this, denote by $\G$ the Zariski closure of $\Gamma$ in 
$\GL_n(\CC)$, and consider the equality 
\[ w(s_1, \dots, s_r, h)= g_i  \]
as an equation in entries $h_{ij}$ of the matrix $h$. It is clear that this 
is a set of polynomial equations. This means that the set of solutions is a Zariski closed subset of $\GL_n(\CC)$. Note also that since the length of $g_i$ is $n_i > \ell(w)$, substituting the identity matrix for $h$ does not yield an equality. In other words, each one of these Zariski closed sets is proper. Now, if 
$\Gamma-A$ is finite, then by adding finitely many points $G$, and hence it
Zariski closure is a union of finitely many proper Zariski-closed sets. This is a contradiction unless $\Gamma$ is finite.  This shows that there are infinitely many 
$S_j= S \cup \{  t_j \}$ that satisfy (1). 

We will now proceed as follows: start with the generating set $S$ as above and let $S_1= S \cup \{ t_1 \} $ be a new generating set satisfying (1). Note that 
$\ell_{S_1}(t_1)=1$, while $\ell_{S}(t_1)>1$, hence $\ell_S$ and $\ell_{S_1}$ are different metrics. 

Now, applying 
what has already been proved to the set $ \{ g_1, \dots, g_m, t_1 \}$ we can find a new generating sense $S_2= S_1 \cup \{ t_2 \} $ with $|S_2|=|S_1|+1$ such that 
 $$ \ell_{ S_2}(g_i)=k_i$$ 
for all $ 1 \le i \le m$, and $j \ge 1$ and $ \ell_{S_2}(t_1)=1$. Note that $\ell_{S_2}(t_2)=1$, while since $t_2 \not\in S_1$ we have 
$\ell_{S}(t_2)>1$ and $\ell_{S_1}(t_2)>1$. In particular, the metric $\ell_{S_2}$ is different from both $\ell_S$ and $\ell_{S_1}$. The argument can now be continued by induction in the same fashion such that $S_j = S \cup \{ t_1, \dots, t_j \}$. Now, using the argument above, we find $t_{ j+1} \not \in S_j$  
such that if $S_{ j+1}= S_j \cup \{ t_{ j+1} \}$ then $ \ell_{ S_{j+1}}(g_1)=k_1, \dots, \ell_{ S_{ j+1}} ( g_m)= k_m$. Now, since $\ell_{ S_{ j+1}}(t_{j+1})=1$, while
$\ell_{S_q}( t_{ j+1} )>1$, we deduce that $\ell_{S_{ j+1} }$ is distinct from all the previously constructed metrics.

\end{proof}

\section{Conclusion} 
In this paper we have shown that one can that group theoretic phenomenon of distortion in arithmetic groups can be used to build a new symmetric-key cryptographic scheme. Arithmetic groups, on the one hand, have the advantage that their elements can be communicated rather easily. On the other hand, the inherent complexity of these groups and their rich subgroup structure has the potential to turn them into fertile ground for cryptographic purposes.  Let us conclude the article with two remarks. 

First note that using the arithmetic group $\SL_n(\ZZ)$ has several advantages over hyperbolic groups, one of which is that elements in $\SL_n(\ZZ)$ can be easily communicated without revealing how they are expressed in terms of generators of the group. This can, of course, also be done for those hyperbolic groups which are linear, but there does not seem to be an obvious way of replicating this for general hyperbolic groups. 
 
Solving the geodesic problem seems to be difficult for arithmetic groups. However, a modification of the above protocol can be made to work if one can solve the {\it modified geodesic problem}.  Given an arbitrary element $g \in \SL_3(\ZZ)$,  the modified geodesic problem find generators $g_1, \dots, g_m \in E$  such that $g= \prod_{1 \le i \le m }  g_i^{\pm 1}$ and $m< C\ell_G(g)$ for a uniform constant $C$.

The question has been studied in \cite{Riley} for the case $G=\SL_n(\ZZ)$, where such a result is established with constant $c(n)=O(n^n)$ which is impractical for real applications. Since the lower bound proven in \cite{Riley} is independent of $n$, one may inquire if the constant can be improved or even made independent of $n$.

\section*{Acknowledgements}
Delaram Kahrobaei is partially supported by a PSC-CUNY grant from the CUNY Research Foundation, the City Tech Foundation, and ONR (Office of Naval Research) grant N00014-15-1-2164. Delaram Kahrobaei has also partially supported by an NSF travel grant CCF-1564968 to IHP in Paris.

\end{document}